\documentclass[11pt]{amsart}

\usepackage[T1]{fontenc}
\usepackage[latin1]{inputenc}
\usepackage[english]{babel}
\usepackage{amsmath,amssymb,amsthm} 
\usepackage{enumitem}
\usepackage{mymacros}

\usepackage[margin=1.5in]{geometry}
\usepackage{color}
\usepackage{hyperref}
\usepackage{bm}

\title{Multivariable versions of a lemma of Kaluza's}
\author{Stefan Richter}
\author{Jesse Sautel}
\address{Department of Mathematics\\ The University of Tennessee\\ Knoxville, TN 37996}
\email{srichter@utk.edu}
\email{jsautel@vols.utk.edu}
\date{\today}
\thanks{Part of this work appeared in the second author's Ph.D. dissertation, which was written under the direction of the first author}
\subjclass[2010]{Primary: 46E22; Secondary 32A05}
\keywords{Kaluza's Lemma, complete Nevanlinna-Pick kernel, renewal sequence}

\newcommand{\la}{\langle}
\newcommand{\ra}{\rangle}
\newcommand{\Bd}{\mathbb B_d}
\newcommand{\F}{\mathbb F}

\begin{document}
\bibliographystyle{plain}
\numberwithin{equation}{section}

\begin{abstract} Let $d\in \N$ and $f(z)= \sum_{\alpha\in \N_0^d} c_\alpha z^\alpha$ be a convergent multivariable power series in $z=(z_1,\dots,z_d)$. In this paper we present two conditions on the positive coefficients $c_\alpha$ which imply that $f(z)=\frac{1}{1-\sum_{\alpha\in \N_0^d} q_\alpha z^\alpha}$ for non-negative coefficients $q_\alpha$. If $d=1$, then both of our results reduce to a lemma of Kaluza's. For $d>1$ we present examples to show that our two conditions are independent of one another. It turns out that functions of the type
$$f(z)= \int_{[0,1]^d} \frac{1}{1-\sum_{j=1}^d t_j z_j} d\mu(t)$$ satisfy one of our conditions, whenever $d\mu(t) = d\mu_1(t_1) \times \dots \times d\mu_d(t_d)$ is a product of probability measures $\mu_j$ on $[0,1]$.

Our results have applications to the theory of Nevanlinna-Pick kernels.
\end{abstract}

\maketitle
\section{Introduction}

In 1928 Theodor Kaluza published the following fact about power series of one complex variable, \cite{Kaluza}:

\begin{thm}\label{Kaluza}(Kaluza's Lemma)   Let $M>0$ and let $\{c_n\}_{n \geq 0}$ be a sequence of positive real numbers with $c_0=1$. Define a sequence of real numbers $\{r_n\}_{n \ge 1}$ by $r_n : = \dfrac{c_n}{c_{n-1}}$ for each $n\in \N$.

If $\{r_n\}_{n \geq 1}$ is a non-decreasing sequence that is bounded above by $M$, then $f(z) := \displaystyle{\sum_{n=0}^\infty c_nz^n}$ converges for all $z\in B(1/M)=\{z\in\C:|z|<1/M\}$ and the coefficients  $\{q_n\}_{n\geq 1}$ defined by
$$f(z)= \frac{1}{1-\sum_{n=1}^\infty q_nz^n}$$
 are all non-negative.\end{thm}

Of course, another way of saying that $\{r_n\}$ is non-decreasing is to say that $\{c_n\}$ is logarithmically convex. But because of Theorem \ref{Kaluza} such sequences have also been called Kaluza sequences in the literature, see e.g. \cite{Kendall}, \cite{Horn},\cite{Shanbhag},\cite{Letac}, or \cite{Cohen}, \cite{Gomilko}. Sequences that satisfy the conclusion of Theorem \ref{Kaluza} have come up in the theory of renewal processes (see e.g. \cite{Kingman},\cite{Pitman}) and in the theory of complete Nevanlinna-Pick reproducing kernels (see e.g. \cite{AMcCpaper}, \cite{AMcC}).

A scaling argument can be used to reduce the statement in Theorem \ref{Kaluza} to the special case where $M=1$.  Then the boundedness  of $r_n$ by 1  easily implies that $f(z)$ is defined and analytic in the unit disc $\D$, and it satisfies $f(0)=1$. Hence $f=\frac{1}{1-g}$ for some function $g$ that is analytic in a neighborhood of 0 with $g(0)=0$. Thus the coefficient sequence $\{q_n\}$ exists, and it is the non-decreasing property of $r_n$ that implies that each $q_n$ is non-negative.
If one knows that the power series converge, then by multiplying with the denominator and adding a term one sees that the relationship between the sequences $c=\{c_n\}_{n\ge 0}$ and $q=\{q_n\}_{n\ge 0}$ is given by
$$f(z)=\sum_{n=0}^{\infty}c_nz^n = 1+ \sum_{n=0}^\infty \left(\sum_{k=0}^n c_{n-k}q_k \right)z^n \ \ \text{ or }$$
$$c_n = \delta_{n,0} +\sum_{k=0}^n c_{n-k}q_k \ \ \ \text{ for all } n\ge 0.$$
This last equation can be written as the discrete renewal equation
\begin{equation}\label{renewal 1d equa} c=\delta_0+c*q,\end{equation} where $\delta_0=\{\delta_{n,0}\}_{n \ge 0}$.
In Section \ref{sec3} we will consider a renewal equation on certain strongly graded monoids (see Equation (\ref{renewal equa})). We will establish a version of Kaluza's Lemma in that context, see Theorem \ref{Monoidkaluza}. The monoids include $\mathbb{F}_d$, the free semigroup on $d$-generators, and the multi-indices $\N_0^d$, see Corollaries \ref{KaluzaEqThm Fd} and \ref{KaluzaEqThm Nd}. By use of a symmetrization  technique the $\mathbb{F}_d$-Corollary will then be used to derive a second version of Kaluza's Lemma on Monoids (Corollary \ref{thm2 general}).

These results lead to two versions of Kaluza's Lemma for functions on the $\ell_1$-balls in $\C^d$, which will be stated now.
Naturally, in order to state our theorems it is most convenient to use multi-index notation. Let $d\in \N$ and let $\alpha=(\alpha_1, \dots, \alpha_d)\in \N_0^d$ be a multi-index. Particular examples of multi-indices are $\bm{0}=(0,\dots,0)$ and for $1\le j\le d$ we use $e_j$ for the multi-index that has 1 in the $j$-th component and 0's elsewhere.

  Furthermore, $|\alpha|=\alpha_1+\dots +\alpha_d$ and $\alpha!= \alpha_1! \cdot \dots \cdot \alpha_d!$. For $z=(z_1, \dots, z_d)\in \C^d$ we write $z^\alpha= z_1^{\alpha_1} \cdot \dots \cdot z_d^{\alpha_d}$, and with this notation for each $n\in \N$ the multinomial formula says that
$$\left(\sum_{j=1}^d z_j\right)^n= \sum_{|\alpha|=n} \frac{|\alpha|!}{\alpha!} z^\alpha.$$

Addition and subtraction of multi-indices are defined componentwise. We also define a partial order on the multi-indices. For $\alpha, \beta\in \N_0^d$ we say $\alpha \le \beta$ if and only if  $\alpha_j \leq \beta_j$ for each integer $j$ with $1 \leq j \leq d.$  We say a sequence of real numbers $\{r_\alpha\}_{\alpha \in \N_0^d}$  is non-decreasing if and only if $r_\alpha \leq r_\beta$ for all pairs of multi-indices with $\alpha \leq \beta$.

 For $\alpha \in \N_0^d, \alpha \ne \textbf{0}$ we define the immediate predecessor set $$P_\alpha= \{\alpha-e_k: 1\le k\le d \text{ and } \alpha_k >0\}.$$ Finally, the $\ell_1$-ball of $\C^d$ is denoted by
$$\mathbb{B}_d^{\ell_1}=\{z = (z_1,\ ...,\ z_d) \ :\ \|z\|_1 := |z_1| + |z_2| + ... + |z_d| < 1\}.$$
\begin{thm} \label{1} Let  $\{c_\alpha\}_{\alpha \in \N_0^d}$ satisfy $c_\textbf{0} = 1$ and $ c_\alpha > 0$ for all  $\alpha \in \N_0^d, \alpha \ne \textbf{0}$.  Define $\{r_\alpha\}_{\alpha \in \mathbb{N}^d_0}$ by
 $r_\textbf{0}  = 0$ and $$r_\alpha = \dfrac{c_\alpha}{\displaystyle{\sum_{\beta\in P_\alpha} c_{\beta}}}, \ \ \ \alpha \neq \bm{0}.$$

If $\{r_\alpha\}_{\alpha \in \mathbb{N}^d_0}$ is non-decreasing and bounded above by 1, then $f(z) = \displaystyle{\sum_{\alpha \in \mathbb{N}^d_0}c_\alpha z^\alpha}$ converges on $\mathbb{B}_d^{\ell_1}$  and there is a sequence of non-negative real numbers $\{q_\gamma\}_{\gamma \in \mathbb{N}_0^d}$ such that
$$f(z)= \frac{1}{1-\displaystyle{\sum_{\gamma \in \mathbb{N}^d_0}q_\gamma z^\gamma}} \text{ on } \mathbb B_d^{\ell_1}.$$
\end{thm}

We note for $\alpha \ne 0$ we have $P_\alpha \ne \emptyset$ and hence the denominator of $r_\alpha$ is strictly positive. Thus, $r_\alpha$ has been well-defined.

\begin{thm}\label{2}  Suppose $\{c_\alpha\}_{\alpha \in \N_0^d}$ satisfies $c_\textbf{0} = 1$ and $ c_\alpha > 0$ for all other $\alpha \in \N_0^d$. If for all $\alpha\in \N_0^d$ and all $i,j\in \N$ with $1\le i,j \le d$ we have
 $c_\alpha \leq \dfrac{|\alpha|!}{\alpha!}$ and
 \begin{equation*}
  \dfrac{c_{\alpha + e_i}c_{\alpha + e_j}}{c_\alpha c_{\alpha + e_i + e_j}} \leq
    \begin{cases}
       \dfrac{|\alpha| + 1}{|\alpha| + 2} \ \ \ \ \ \ \ \text{ if } i\ne j\\
       \\
       \dfrac{(|\alpha|+1)(\alpha_i + 2)}{(|\alpha|+2)(\alpha_i+1)} \text{ if } i= j
    \end{cases}
\end{equation*}
    then $f(z) = \displaystyle{\sum_{\alpha \in \mathbb{N}^d_0}c_\alpha z^\alpha}$ converges on $\mathbb{B}_d^{\ell_1}$ and there is a sequence of non-negative real numbers $\{q_\gamma\}_{\gamma \in \mathbb{N}_0^d}$ such that
$$f(z)= \frac{1}{1-\displaystyle{\sum_{\gamma \in \mathbb{N}^d_0}q_\gamma z^\gamma}} \text{ on } \mathbb B_d^{\ell_1} .$$ \end{thm}
 It appears that the hypothesis in Theorem \ref{1} will be difficult to check for many examples, as the definition of the auxiliary sequence $\{r_\alpha\}$ is fairly complicated. Nevertheless, if $\{r_\alpha\}_{\alpha\in \N_0^d}$ is any non-decreasing sequence that is bounded above by 1 and satisfies $r(\bm{0})=0$, then one can use the conditions $c_{\bm{0}}=1$ and $c_\alpha = r_\alpha (\sum_{\beta\in P_\alpha} c_\beta)$ for $\alpha \ne \bm{0}$ to inductively define $c_\alpha$. That is, because for $|\alpha|=n$ we have $|\beta|=n-1$ for all $\beta\in P_\alpha$. This leads to many examples, where the conclusion of the Theorem may be non-obvious. Furthermore, in Section \ref{SecExample} we will use this observation to  construct an example that satisfies the hypothesis of Theorem \ref{1}, but not the one of Theorem \ref{2}. In Section \ref{SecExample2} we will show that $$f(z)=\int_{[0,1]^d} \frac{1}{1-\sum_{j=1}^d t_j z_j} d\mu(t)$$ satisfies the conditions of Theorem \ref{2}, whenever $d\mu(t) = d\mu_1(t_1) \times \dots \times d\mu_d(t_d)$ is a product of probability measures $\mu_j$ on $[0,1]$. We will also see that for $d\mu(s,t)=dsdt$ the function $f$ does not satisfy the hypothesis of Theorem \ref{1}. Thus, Theorems \ref{1} and \ref{2} are independent of one another.

 In Section \ref{sec:repkernels} we will note that our theorems lead to sufficient conditions for  reproducing kernels to be so-called complete Nevanlinna-Pick kernels. We will then give an example to illustrate what the underlying Hilbert function spaces may look like.


\section{Preliminaries on graded monoids}\label{Sec2}
Recall that a semi-group is a set $M$ with an operation $M\times M\to M, (x,y)\to xy$ that satisfies the associative law.
A monoid is a semi-group $M$ with identity $e$ satisfying $ex=xe=x$ for all $x\in M$. We will say a monoid $M$ is  strongly graded (by $\N_0$), if there is a collection of mutually disjoint subsets $\{M_n\}_{n\in \N_0}$ such that $M=\bigcup_{n=0}^\infty M_n$ and $M_nM_m= M_{n+m}$ for all $n,m \in \N_0$. On strongly graded monoids we can define the length function $|\cdot|:M \to \N_0$ by $|w|=n$, if $w\in M_n$. We note that the condition $M_nM_m=M_{n+m}$ implies that this function is a homomorphism.

We will now assume that $M$ is a strongly graded monoid satisfying two further conditions:
\begin{description}
  \item[TK] the length function has a trivial kernel: $|w|=0$, if and only if $w=e$ and
  \item[RC]  the following (right) cancellation law holds: if $x,y\in M_1$ and $w\in M$ such that $xw=yw$, then $x=y$.
\end{description}
By definition  condition TK is  equivalent to saying that $M_0=\{e\}$, and it implies for strongly graded monoids that whenever $w\in M$ and $n\in \N_0$, then $w\in M_n$ if and only if there are $x_1, \dots , x_n \in M_1$ such that $w=x_1 \dots x_n$. Of course, in general such a representation may not be unique.

\begin{exa} $\N_0^d$, the multi-indices. In this case we take $e=\textbf{0}$, the semi group operation is componentwise addition of multi-indices, and for each $n \in \N_0$ we can take $M_n=\{\alpha\in \N_0^d: |\alpha|=n\}$, where $|\cdot|$ is the usual length function of multi-indices. It is easily checked that this defines a strong grading that satisfies conditions TK and RC as well. We caution that the commutativity of the semigroup operation implies that the right cancellation law only holds with elements in $x,y\in M_1$.
\end{exa}
\begin{exa} $M=\N_0^{\aleph_0}=\{(\alpha_1, \alpha_2, \dots ): \alpha_j\in \N_0 \text{ and } \exists \ N\ :  \alpha_j=0 \text{ for all } j\ge N\}$. Again we use componentwise addition as the semi-group operation. Since all sequences are eventually 0, the length function is well-defined as the sum of the components. Then $M_1=\{e_j: j\in \N\}$ is an infinite set.
\end{exa}

\begin{exa} Let $A$ be a non-empty set, and let $\F(A)$ be the free monoid over the alphabet $A$. That is $\F(A)$ is the collection of all words of finite length that can be formed with letters from $A$ together with the empty word $\bm{0}$. Thus
$$\F(A)= \{\bm{0}\} \cup \{w_{1}w_{2} \dots w_n: n\in \N \text{ and } w_1, \dots, w_n \in A\}.$$
The semi-group operation is concatenation of words $(u,v) \to uv$ and $e=\bm{0}$.
If $w\in \F(A)$, then it has a unique representation of the type $w=w_{1}w_{2} \dots w_n$, where $w_1, \dots, w_n \in A$ and we can define the length of $w$ by $|w|=n$. Then we set $M_n=\{w\in \F_d:|w|=n\}$ and it is now easy to see that $\{M_n\}_{n\in \N_0}$ induces a strong grading on $\F(A)$ that satisfies conditions TK and RC.

If the cardinality of $A$ equals $d\in \N$, then $F(A)$ is isomorphic to $\F_d=\F(\{1,2, \dots, d\})$, the free monoid on $d$ generators.
\end{exa}
\begin{exa}If $M$ and $N$ are strongly graded monoids that satisfy TK and RC, then so is $M\times N$. In this case the semi group operation is defined by $(x,v)(y,w)=(xy,vw)$, the unit is $e=(e_M,e_N)$, and the grading is given by $$(M\times N)_n= \bigcup_{k=0}^n M_k\times N_{n-k}.$$
\end{exa}
\begin{proof} It is easy to see that $M\times N$ is strongly graded with $(M\times N)_0=\{e\}$. We check the right cancellation law RC:

 Suppose $(x,v)\in M\times N$ and $(y_1,w_1), (y_2,w_2)\in (M\times N)_1$ such that
$(y_1,w_1)(x,v)=(y_2,w_2)(x,v)$. Then $y_1x=y_2x$ and $w_1v=w_2v$. Since $(y_1,w_1)\in (M\times N)_1=(M_1\times \{e_N\})\cup (\{e_M\}\times N_1)$ we have $y_1=e_M$ or $w_1=e_N$.

If $y_1=e_M$, then $w_1v=w_2v$ and hence $1+|v|_N=|w_1v|_N=|w_2|_N+|v|_N$. Hence $|w_2|_N=1$ and  by the right cancellation law in $N$ we have $w_1=w_2$. Furthermore,  $(y_2,w_2)\in (M\times N)_1$ now implies $y_2=e_M$. Hence $(y_1,w_1)= (y_2,w_2)$. The case $w_1=e_N$ follows by symmetry.
\end{proof}
If $M$ is a strongly graded monoid, then the universal property of the free monoid $\F(M_1)$ implies that there exists a homomorphism \begin{equation}\label{universal} \varphi:\F(M_1)\to M\end{equation} such that $\varphi(x)=x$ for each $x\in M_1$. The strong grading of $M$ implies that $\varphi$ is onto, and hence $M$ is isomorphic to the quotient  $\F(M_1)/\sim$, where the equivalence relation $\sim$ is defined by $w\sim v$, if and only if $\varphi(w)=\varphi(v)$.  We note that it is easy to check that $\varphi$ (and hence $\sim$) preserves the length functions, $|\varphi(w)|_M=|w|_{\F(M_1)}$ for all $w\in \F(M_1)$. If $M=\N_0^d$, then one can take $\varphi: \F_d \to \N_0^d$ is given by the usual correspondence between words $w=w_1 w_2 \dots w_n\in \F_d$ and multi-indices $\alpha=(\alpha_1, \dots, \alpha_d)\in \N_0^d$: $\varphi(w)=\alpha$, if $\alpha_k=$ the number of times the letter $k$ is used in the word $w$.
We will use the map $\varphi$ in Section \ref{sec:thm2} in order to prove Theorem \ref{2}.

We define  a (right) partial order on such an $M$ by saying $x\le y$, if there is $z\in M$ such that $xz=y$. Furthermore, we say a function $r: M\to \R$ is (right) non-decreasing, if $r(x)\le r(y)$, whenever $x\le y$. For later reference we record a simple lemma. We omit the easy inductive proof.

\begin{lem} \label{right increasing} If $r:M\to \R$ satisfies $r(e)=0$, then $r$ is right non-decreasing, if and only if $r(xv)\le r(xvy)$ for all $x,y\in M_1$ and all $v\in M$.
\end{lem}

For $w\in M$ we write $D_w=\{(u,v)\in M\times M: uv=w\}$, the decompositions of $w$, and  $P_w=\{v\in M: xv=w \text{ for some }x\in M_1\}$, the immediate predecessors of $w$.

\begin{lem} \label{Lem1} If $x,y,w\in M$, then $xy\in P_w$, if and only if there is $u\in M$ such that $uy=w$ and $x\in P_u$. Such a $u$ is necessarily unique and it satisfies $u \le w$.

Hence for $w\in M$ we have
$$A_w=\{(x,y): xy\in P_w\}=\{(x,y): \exists ! \ u\le w \text{ such that } (u,y)\in D_w \text{ and }x\in P_u\}.$$
\end{lem}
\begin{proof} If $xy\in P_w$, then  there is a  $v\in M_1$ such that $vxy=w$. Setting $u=vx$ we have $x\in P_u$ and $uy=w$. Conversely, if $u\in M$ with $uy=w$ and $x\in P_u$, then there is $v\in M_1$ such that $vx=u$ and hence $vxy=w$. This implies $xy\in P_w$.

It is clear that any such $u$ satisfies $u\le w$. Finally, assume that $u_1x=w=u_2 x$ and $y\in P_{u_1}\cap P_{u_2}$, then there are $v_1, v_2\in M_1$ such that $v_j y= u_j$ for $j=1,2$ and hence $v_1yx=u_1x=u_2x=v_2yx$. Thus, the cancellation law implies that $v_1=v_2$ and hence $u_1=u_2$.
\end{proof}

We say that $f:M\to \C$ is locally summable, if $\sum_{|x|\le n}|f(x)|<\infty$ for each $n\in \N_0$. Of course, if $M_1$ is finite, then each $M_n$ must be finite and hence every function is locally summable.
If $f,g:M\to \C$ are locally summable, then for each $w\in M$ we let $n=|w|$ and we have $$\sum_{(u,v)\in D_w} |f(u)g(v)| \le \sum_{|u|\le n}\sum_{|v|\le n} |f(u)g(v)| <\infty.$$
Hence we can define the convolution of $f$ and $g$ by $$(f*g)(w)=\sum_{(u,v)\in D_w} f(u)g(v).$$ Let $\delta_e: M\to \C$ be defined by $\delta_e(x)=1$, if $x=e$ and $\delta_e(x)=0$ otherwise.

Given a function $f: M\to \R$ we are interested in solutions $g$ of the renewal equation \begin{equation}\label{renewal equa}f=\delta_e + f*g.\end{equation}

\begin{lem}\label{Lem2} For every locally summable $f: M\to \R$ with $f(e)=1$ the renewal equation (\ref{renewal equa}) has a unique locally summable solution $g:M\to \R$. The solution satisfies $g(e)=0$ and
\begin{equation}\label{induct}g(w)= f(w) - \sum_{(u,v)\in D_w, v\ne w} f(u)g(v) \ \  \text{ if } w\ne e. \end{equation}
\end{lem}
\begin{proof} First suppose that $g$ is a locally summable solution to the renewal equation. Since $D_e=\{(e,e)\}$ we have $1=f(e)=\delta_e(e)+ (f*g)(e)=1+f(e)g(e)=1+g(e)$ and hence $g(e)=0$. Note that if $uw=w$, then condition TK implies that $u=e$. Thus, if $w\ne e$, then $D_w=\{(u,v): uv=w, v\ne w\} \cup \{(e,w)\}$ is a disjoint union and hence
$$ f(w)= 0+(f*g)(w)=\sum_{(u,v)\in D_w, v\ne w} f(u)g(v) + g(w).$$
This implies that $g$ satisfies (\ref{induct}). It also shows that there is at most one solution. Indeed, we note that if $(u,v)\in D_w$ with $v\ne w$, then $|w|=|u|+|v|>|v|$. Thus, if $w\in M_n$ for some $n\in \N$, then all $v\ne w$ such that there is $u\in M$ with $(u,v)\in D_w$ must satisfy $v\in \bigcup_{j=0}^{n-1} M_j$ and the identity (\ref{induct}) can be used to define $g$ inductively (starting with $g(e)=0$). Uniqueness follows, as does existence, we just need to make sure that $g$ turns out to be locally summable (in order for the convolution to be well-defined).

 If $w\in M_1$, then $D_w=\{(e,w),(w,e)\}$, hence we have
$$g(w)= f(w)-f(w)g(e)=f(w).$$
 This together with $g(e)=0$ and the hypothesis that $f$ is locally summable implies $\sum_{|w|\le 1} |g(w)|<\infty$.

Suppose $n\ge 1$, and for all $w\in M$ with $|w|\le n$ we have already found $g(w)\in \R$ that satisfy $$\sum_{|w|\le n}|g(w)|<\infty$$ and such that $g(w)$ satisfies (\ref{induct}). Then
\begin{align*}\sum_{|w|={n+1}} |f(w)-\sum_{(u,v)\in D_w, v\ne w} & f(u)g(v)|\\ &\le \sum_{|w|= n+1} |f(w)| + \sum_{|u|\le n+1,|v|\le n}|f(u)g(v)|<\infty.
\end{align*}
Thus, for each $w\in M_{n+1}$ we can use $\ref{induct}$ to define $g(w)\in \R$ and this will satisfy $$\sum_{|w|= n+1}|g(w)|<\infty.$$
This concludes the inductive step to show that we can define a locally summable function $g$ that satisfies the renewal equation.
\end{proof}


\section{Kaluza's lemma for strongly graded Monoids}\label{sec3}
\begin{thm}\label{Monoidkaluza} Let $M$ be a strongly graded monoid that satisfies the conditions TK and RC of Section \ref{Sec2}. Let $f: M \to (0,\infty)$ be locally summable with $f(e)=1$. Define a function $r: M\to [0,\infty)$ by $r(e)=0$ and $$r(w)= \frac{f(w)}{\sum_{v\in P_w} f(v)} \text{ if } w\ne e.$$

If $r$ is right non-decreasing, then the solution $g$ of the renewal equation \ref{renewal equa} is non-negative.
\end{thm}
Note that if $w\ne e$, then $P_w\ne \emptyset$. Hence the positivity of $f$ implies that $r$ is well-defined in all cases.
\begin{proof} We assume that $r$ is right non-decreasing. Since $M=\bigcup_{n=0}^\infty M_n$ we can use induction on $n$ and the inductive formula from Lemma \ref{Lem2} to show that $g$ is non-negative. The case $n=0$ follows immediately since $M_0=\{e\}$ and $q(e)=0$. We also check the case $n=1$. If $w\in M_1$, then we saw in the proof of Lemma \ref{Lem2} that $g(w)=f(w)>0$.

Now assume that $n\in \N$ and that $g(v)\ge 0$ for all $v\in M$ with $|v|\le n$.
Let $|w|={n+1}$. We have to show that $g(w)\ge 0$.

Note for later that for all $u\in M, u \ne e$ we have
\begin{equation}\label{rf}f(u)= r(u)\sum_{x\in P_u}f(x).\end{equation}
For $u\in P_w$ set $p(u)=\frac{f(u)}{\sum_{v\in P_w} f(v)}$. Then $\sum_{u\in P_w} p(u)=1$ and $r(w)f(v)=f(w)p(v)$.

If $v\in P_w$, then since $|w|\ge 2$  we have $v\ne e$ and hence we can use the renewal equation to conclude that
\begin{equation}\label{renewal}f(v)=\sum_{(x,y)\in D_v}f(x)g(y).\end{equation}

Hence by Lemma \ref{Lem2}
\begin{align*}g(w)
&= f(w)- \sum_{(u,y)\in D_w, y\ne w}f(u)g(y)\\
&=\sum_{v\in P_w} f(w)p(v)- \sum_{(u,y)\in D_w, y\ne w}f(u)g(y)\\
&=\sum_{v\in P_w} r(w)f(v) - \sum_{(u,y)\in D_w, y\ne w}r(u)\sum_{x\in P_u}f(x)g(y) \ \ \  \text{ by (\ref{rf})}\\
&=\sum_{v\in P_w} r(w) \sum_{(x,y)\in D_v}f(x)g(y) - \sum_{(u,y)\in D_w, y\ne w}r(u)\sum_{x\in P_u}f(x)g(y) \ \ \text{ by (\ref{renewal})}\\
&=\sum_{(x,y): xy\in P_w}r(w)f(x)g(y)-  \sum_{(u,y)\in D_w, y\ne w, x\in P_u}r(u)f(x)g(y)
\end{align*}
Note that both terms in this difference are positive term sums that are finite by the identities used. Hence there is no issue with rearranging the sums.
Let $A_w$ denote the set from Lemma \ref{Lem1}, and note that if $(x,y)\in A_w$, then $y\ne w$. Hence  Lemma \ref{Lem1} implies that
$$g(w)= \sum_{(x,y)\in A_w} (r(w)-r(u))f(x)g(y),$$ where we have written $u=u(x,y)$ for the element in $M$ whose existence and uniqueness was proven in Lemma \ref{Lem1}. Hence $g(w)\ge 0$ by the inductive hypothesis and the fact that $r$ is right non-decreasing.
\end{proof}

We collect two corollaries for the special cases where $M=\F_d$ and $M=\N_0^d$.

\begin{cor} \label{KaluzaEqThm Fd} Let $d\in \N$, $f:\F_d\to (0,\infty)$ with $f(\bm{0})=1$. If also $$\frac{f(av)}{f(v)} \le \frac{f(avb)}{f(vb)} \ \ \text{ for all } a,b, v\in \F_d, |a|=|b|=1,$$
then the solution $g: \F_d\to \R$ to the renewal equation (\ref{renewal equa}) is non-negative.
\end{cor}
\begin{proof} By Theorem \ref{Monoidkaluza} and Lemma \ref{right increasing} we have to show that $$\frac{f(av)}{\sum_{u\in P_{av}} f(u)} \le \frac{f(avb)}{\sum_{u\in P_{avb}} f(u)}$$ for all $a,b, v\in \F_d, |a|=|b|=1.$ But in $\F_d$ we have $P_{av}=\{v\}$ and $P_{avb}=\{vb\}$. The result follows.
\end{proof}
\begin{cor} \label{KaluzaEqThm Nd} Let $d\in \N$, $f:\N_0^d \to (0,\infty)$ with $f(\bm{0})=1$. Define $r:\N_0^d\to [0,\infty)$ by $r(\bm{0})=0$ and $$r(\alpha)=\frac{f(\alpha)}{\sum_{1\le k\le d, \alpha_k>0}f(\alpha-e_k)} \ \ \text{ if }\alpha \ne \bm{0}.$$

If $r(\alpha) \le r(\beta)$ for all $\alpha, \beta\in \N_0^d$ with $\alpha_j\le \beta_j$ for all $j=1, \dots, d$, then
the solution $g: \N_0^d\to \R$ to the renewal equation (\ref{renewal equa}) is non-negative.
\end{cor}

\begin{proof} This also follows from Theorem \ref{Monoidkaluza} by noting that for $\alpha \in \N_0^d$ we have $P_\alpha=\{\alpha-e_k: 1\le k\le d \text{ and } \alpha_k>0\}$, and that the partial order that we defined on the multi-indices agrees with the partial (right) ordering that we defined on strongly graded monoids. \end{proof}


\section{The  proof of Theorem \ref{1}}\label{SecExample}

In this section we assume that $\{c_\alpha\}_{\alpha\in \N_0^d}$ is a given sequence of positive reals with $c_{\bm{0}}=1$, and that $r_\alpha$ is defined as in Theorem \ref{1}.
\begin{lem}\label{lem:Sum to n} If $n\ge 2$ and $z=(z_1,\dots, z_d)\in \C^d$, then
$$(1-\sum_{j=1}^d z_j) \sum_{|\alpha|\le n} c_\alpha z^\alpha = 1- \sum_{0<|\alpha|\le n} (1-r_\alpha)z^\alpha \sum_{\beta\in P_\alpha} c_\beta - \sum_{|\alpha|=n+1}z^\alpha \sum_{\beta\in P_\alpha}c_\beta. $$ \end{lem}
\begin{proof} We have
\begin{equation}\label{equ:sum n 1}\sum_{|\alpha|\le n}c_\alpha z^\alpha =1+\sum_{0<|\alpha|\le n}r_\alpha  z^\alpha \sum_{\beta\in P_\alpha} c_\beta \end{equation} and
\begin{equation}\label{equ:sum n 2}(\sum_{j=1}^d z_j)\sum_{|\alpha|\le n}c_\alpha z^\alpha = \sum_{|\beta|\le n}\sum_{j=1}^d c_{\beta}z^{\beta+e_j}= \sum_{0<|\alpha|\le n+1} z^\alpha\sum_{\beta\in P_\alpha} c_\beta.\end{equation} The proof now follows by simple algebra.\end{proof}

\begin{lem}\label{lem:b alpha} If $r_\alpha \le 1$ for all $\alpha \in \N_0^d$, then $$\sum_{\alpha \in \N_0^d} |c_\alpha z^\alpha|<\infty$$ for all $z\in \mathbb{B}_d^{\ell_1} $. Furthermore, the non-negative coefficients $b_\alpha$ defined for $\alpha \ne \bm{0}$ by $b_\alpha= \frac{c_\alpha}{r_\alpha}(1-r_\alpha)=(1-r_\alpha) \sum_{\beta\in P_\alpha} c_\beta$ satisfy
\begin{equation}\label{equ:b alpha identity}\sum_{\alpha \in \N_0^d} c_\alpha z^\alpha = \frac{1-\sum_{\alpha \in \N_0^d, \alpha\ne \bm{0}} b_\alpha z^\alpha}{1-\sum_{j=1}^d z_j} \ \ \text{ on }\mathbb{B}_d^{\ell_1}.\end{equation}
\end{lem}
\begin{proof} Let $z\in \mathbb{B}_d^{\ell_1}$ and write $x=(|z_1|, |z_2|, \dots, |z_d|)$, then $\|z\|_1=\|x\|_1=\sum_{j=1}^dx_j<1$. Thus, for $n\in \N$ we have by the hypothesis that  $1-r_\alpha \ge 0$ and by Lemma \ref{lem:Sum to n} applied with the point $x$ that
$$\sum_{|\alpha|\le n} |c_\alpha z^\alpha|= \sum_{|\alpha|\le n} c_\alpha x^\alpha \le \frac{1}{1-\|z\|_1}<\infty.$$
Hence the series converges absolutely on $\mathbb{B}_d^{\ell_1}$. Thus, we can let $n\to \infty$ in equations \ref{equ:sum n 1} and \ref{equ:sum n 2} of the proof of Lemma \ref{lem:Sum to n} and obtain
$$\sum_{\alpha \in \N_0^d}c_\alpha z^\alpha = 1+\sum_{\alpha \in \N_0^d, \alpha \ne \bm{0}}r_\alpha  z^\alpha \sum_{\beta\in P_\alpha} c_\beta $$ and
$$ (\sum_{j=1}^d z_j)\sum_{\alpha \in \N_0^d}c_\alpha z^\alpha= \sum_{\alpha \in \N_0^d, \alpha \ne \bm{0}} z^\alpha\sum_{\beta\in P_\alpha} c_\beta.$$
The identity \ref{equ:b alpha identity} follows by subtracting the two equations and dividing by $(1-\sum_{j=1}^dz_j)$.\end{proof}
\begin{proof}[Proof of Theorem \ref{1}]
We now assume that $r_\alpha \le 1$ for all $\alpha\in \N_0^d$ and that $\{r_\alpha\}_{\alpha\in \N_0^d}$ is non-decreasing. By Lemma \ref{lem:b alpha} and the general theory of holomorphic functions of several complex variables it follows that $f(z)= \sum_{\alpha\in \N_0^d}c_\alpha z^\alpha$ is holomorphic in $\mathbb{B}_d^{\ell_1}$, and it satisfies $f(0)=1$. Then $g(z)=1-1/f(z)$ is holomorphic in a neighborhood of $0$ and hence it has a power series representation $g(z)=\sum_{\beta\in \N_0^d} q_\beta z^\beta$. The functions $f$ and $g$ satisfy the identity $$f(z)= 1+ (fg)(z).$$ If we write $c=\{c_\alpha\}_{\alpha\in \N_0^d}$ and $q=\{q_\beta\}_{\beta \in \N_0^d}$, then the sequence of power series coefficients of the function $fg$ is given by $c*q$. Thus, we see that the coefficient sequences satisfy the renewal equation $c= \delta_0+ c*q$. Hence the non-negativity of the terms of the sequence $q$ follows from Corollary \ref{KaluzaEqThm Nd}.
\end{proof}

We will now construct the promised example to show that for $d\ge 2$ Theorem \ref{1} does not follow from Theorem \ref{2}.
\begin{exa} Let $d\in \N$, $d\ge 2$ and $0<a, b <1$ and define $\{r_\alpha\}_{\alpha\in \N_0^d}$ by $r_{\bm{0}}=0$, $r_{e_1}=a$, $r_{e_2}=b$ and $r_\alpha=1$ for all other $\alpha\in \N_0^d$. Then  $r_\alpha \le 1$ for all $\alpha$ and $\{r_\alpha\}_{\alpha\in \N_0^d}$ is non-decreasing.

Inductively define $\{c_\alpha\}_{\alpha\in \N_0^d}$ by $c_{\bm{0}}=1$ and if $|\alpha|\ne 0$ set $c_\alpha=r_\alpha \sum_{\beta\in P_\alpha}c_\beta$. Then $\{c_\alpha\}_{\alpha\in \N_0^d}$ satisfies the hypothesis of Theorem \ref{1}.

Next compute
$$c_{e_1}=r_{e_1}c_{\bm{0}}=a, \ \ c_{e_2}=r_{e_2}c_{\bm{0}}=b,\ \ c_{2e_1}=r_{2e_1}c_{e_1}=a \ \ \text{ and }$$
$$ c_{e_1+e_2}=r_{e_1+e_2}(c_{e_1}+c_{e_2})=a+b, \ \ c_{2e_1+e_2}=r_{2e_1+e_2}(c_{e_1+e_2}+c_{2e_1})=2a+b.$$

The sequence $\{c_\alpha\}_{\alpha\in \N_0^d}$ does not satisfy the hypothesis of Theorem \ref{2} whenever $a\ne b$. Without loss of generality we assume $a<b$. Then we take $\alpha=e_1 $, $i=1, j=2$ in the hypothesis of Theorem \ref{2}:
$$\frac{c_{\alpha+e_1}c_{\alpha+e_2}}{c_\alpha c_{\alpha+e_1+e_2}}=\frac{c_{2e_1}c_{e_1+e_2}}{c_{e_1}c_{2e_1+e_2}}=\frac{a+b}{2a+b}> \frac{2}{3}= \frac{|\alpha|+1}{|\alpha|+2}.$$

This shows that the hypothesis of Theorem \ref{2} is not met.

We can also calculate the coefficients $b_\alpha$ of Lemma \ref{lem:b alpha}. Indeed, whenever $r_\alpha=1$, then $b_\alpha=0$. Hence we only calculate $b_{e_1}=1-a$ and $b_{e_2}=1-b$, and from Lemma \ref{lem:b alpha} we obtain
$$f(z)= \sum_\alpha c_\alpha z^\alpha = \frac{1-(1-a)z_1-(1-b)z_2}{1-\sum_{j=1}^dz_j}.$$
\end{exa}

The calculation just done highlights that it may sometimes be useful to start with the  sequence $b= \{b_\alpha\}_{\alpha\in \N_0^d}$ of non-negative coefficients such that $b_{\bm{0}}=0$ and $\sum_{\alpha\in \N_0^d}b_\alpha |z^\alpha| <1$ for all $z\in \mathbb{B}_d^{\ell_1}$, and then one may want to compute the sequence $c=\{c_\alpha\}_{\alpha\in \N_0^d}$ satisfying
$$f(z)=\sum_{\alpha\in \N_0^d} c_\alpha z^\alpha= \frac{1-\sum_{\alpha } b_\alpha z^\alpha}{1-\sum_{j=1}^d z_j}.$$ We note that one can express $c$ in terms of $b$ by
\begin{equation}\label{c from b}c_\alpha=\frac{|\alpha|!}{\alpha!}-\sum_{\beta\le \alpha}\frac{|\alpha-\beta|!}{(\alpha-\beta)!}b_\beta.\end{equation}
In order to prove this, write $B(z)=\sum_{\alpha\in \N_0^d } b_\alpha z^\alpha$ and  $D=\{d_\alpha\}_{\alpha\in \N_0^d}$, $d_\alpha=\frac{|\alpha|!}{\alpha!}$. By the multinomial theorem we have
 $$D(z)=\frac{1}{1-\sum_{j=1}^d z_j}=\sum_{n=0}^\infty (\sum_{j=1}^d z_j)^n= \sum_{n=0}^\infty \sum_{|\alpha|=n} d_\alpha z^\alpha.$$
The hypothesis can be written as $f=(1-B)D=D-DB$, and thus the sequences of power series coefficients satisfy $c=d-d*b$, which is exactly (\ref{c from b}).


\section{The proof of Theorem \ref{2}}\label{sec:thm2}
Suppose now that $M$ is a strongly graded monoid that satisfies the conditions TK and RC. Let $e$ be its identity. We will now use the homomorphism $\varphi: \F(M_1) \to M$ that established that $M$ is a quotient of $\F(M_1)$ (see (\ref{universal})). As a matter of convenience for notation we will use letters like $u,v,w $ to denote elements in $\F(M_1)$ (words) and $\alpha, \beta, \gamma$ for elements in $M$ (think e.g. multi-indices).

 Define a function on $M$ by
$$\ N(\alpha)= \mathrm{card}\ \varphi^{-1}(\{\alpha\}), \ \ \alpha\in M.$$

We will now assume that $M$ satisfies the additional hypothesis that $N$ only takes finite values.

For example, that will always be the case, if $M_1$ is finite, say $\mathrm{card }\ M_1=d$. Since $\varphi$ preserves the length function we have $|w|=|\alpha|$ for every $w\in \varphi^{-1}(\{\alpha\})$ and hence $N(\alpha)\le d^{|\alpha|}$.
It is also true for $M=\N_0^{\aleph_0}$. In that case $\varphi(w)$ equals the multi-index associated with the string $w$, i.e. if $w=x_1x_2 \dots x_n$ for $x_k\in \N$, then $\varphi(w)_j= $ the number of times that the integer $j$ shows up among the integers $x_1, x_2, \dots, x_n$. Thus, if $\alpha\in \N_0^{\aleph_0}$ is supported in $\{1,\dots, d\}$, then as above $N(\alpha)\le d^{|\alpha|}$.

For $f: \F(M_1) \to \R$ define the "symmetrization of $f$" by $Sf: M\to \R$, $$Sf(\alpha)=\sum_{w\in \varphi^{-1}(\{\alpha\})} f(w).$$
\begin{lem} Let $f, g:\F(M_1)\to \R$ be locally summable. Then $Sf$ and $Sg$ are locally summable and
$$S(f*g)=(Sf)*(Sg).$$
\end{lem}

\begin{proof} The finiteness assumption $N(\alpha)<\infty$ and the fact that $\varphi$ preserves the length functions easily implies that $Sf$ is locally summable, whenever $f$ is.

Let $\alpha \in M$, then
\begin{align*}(Sf*Sg) (\alpha)&= \sum_{(\beta,\gamma)\in D_\alpha} (Sf)(\beta)(Sg)(\gamma)\\
&= \sum_{(\beta,\gamma)\in D_\alpha} \sum_{\varphi(u)=\beta}\sum_{\varphi(v)=\gamma}f(u)g(v)\\
&=\sum_{\varphi(u)\varphi(v)=\alpha} f(u)g(v)\\
&=\sum_{uv \in \varphi^{-1}(\{\alpha\})} f(u)g(v)
\end{align*}
Furthermore, if $w\in \varphi^{-1}(\{\alpha\})$, then $(f*g)(w)= \sum_{(u,v)\in D_w}f(u)g(v)$ and hence
$$S(f*g)(\alpha)= \sum_{w\in \varphi^{-1}(\{\alpha\}) }\sum_{(u,v)\in D_w}f(u)g(v)= \sum_{uv \in \varphi^{-1}(\{\alpha\})} f(u)g(v)=(Sf*Sg) (\alpha).$$
\end{proof}

\begin{thm}\label{2nd renewal pos} Let $c:M\to \R$ with $c(0)=1$ and define $f:\F(M_1)\to \R$ by  $f(w)=\frac{c(\varphi(w))}{N(\varphi(w))}$. Note that $f(\bm{0})=1$. Let $g:\F(M_1)\to \R$ be the unique solution to the $\F(M_1)$-renewal equation $f=\delta_{\bm{0}}+ f * g$.

Then $q=Sg$ satisfies the $M$-renewal equation $c=\delta_{e} + c*q$.
\end{thm}
\begin{proof} Note that the definition of $f$ implies that for each $\alpha \in M$ $f$ is constant on $\varphi^{-1}(\{\alpha\})$. Thus it is easily seen that $Sf=c$. Hence by the linearity of $S$ and the previous Lemma
$$c=Sf=S\delta_{\bm{0}} +S(f*g)=\delta_e + Sf * Sg= \delta_e+ c*q.$$
\end{proof}
From the definition of $q=Sg$ it follows that $q$ is non-negative, whenever $g$ is non-negative. If $\mathrm{card}\ M_1=d<\infty$ we can thus use Corollary \ref{KaluzaEqThm Fd} give a sufficient condition for $q$ to be non-negative. We write that out in terms of the function $c$ and obtain
\begin{cor} \label{thm2 general} Let $M$ be a strongly graded monoid that satisfies conditions TK and RC and assume that $\mathrm{card}\ M_1=d<\infty$. Let $c:M\to (0,\infty)$ be a function with $c(e)=1$. If for all $\alpha,\beta,\gamma \in M$ with $\alpha, \gamma\in M_1$ we have
$$\frac{c(\alpha \beta)c(\beta \gamma)}{c(\beta)c(\alpha \beta \gamma)} \le \frac{N(\alpha \beta)N(\beta \gamma)}{N(\beta)N(\alpha \beta\gamma)},$$
then the unique solution $q:M\to \R$ to the renewal equation $c=\delta_e+c*q$ is non-negative.
\end{cor}

We note that there is a little bit of room to play in this set-up. In Theorem \ref{2nd renewal pos} we used the most natural way to define $f:\F(M_1)\to \C$ from a given $c:M\to \C$ so that $c=Sf$. A different choice of $f$ could lead to a different version of Corollary \ref{thm2 general} and of Theorem \ref{2}.

\begin{proof}[Proof of Theorem \ref{2}]
We start by showing that $c_\alpha \le \frac{|\alpha|!}{\alpha!}$ implies the absolute convergence of the power series $\sum_\alpha c_\alpha z^\alpha$ in $\mathbb{B}_d^{\ell_1}$. The proof is similar to the earlier argument. If $z\in \mathbb{B}_d^{\ell_1}$, then we set $x=(|z_1|, \dots, |z_d|)$ and observe $\|z\|_1=\sum_{j=1}^d x_j$ and
$$\sum_\alpha |c_\alpha z^\alpha| = \sum_\alpha c_\alpha x^\alpha\le \sum_\alpha  \frac{|\alpha|!}{\alpha!} x^\alpha= \frac{1}{1-\|z\|_1}<\infty.$$
Thus, the power series for $f$ converges and hence as in the proof of Theorem \ref{1} the coefficients $q_\alpha$ exist and are given by the solution to the renewal equation, and we will finish the proof by showing that the second condition in Theorem \ref{2} reduces to the hypothesis of Corollary \ref{thm2 general} when $M=\N_0^d$.

In that case for a multi-index $\alpha$ the quantity $N(\alpha)$ equals the number of words in $\F_d$ that give rise to the multi-index $\alpha$, i.e. $N(\alpha)=\frac{|\alpha|!}{\alpha!}$. The set $M_1$ equals $\{e_1, \dots, e_d\}$ and the semi-group operation is addition, hence by Corollary \ref{thm2 general} the condition equals
$$ \dfrac{c_{\beta + e_i}c_{\beta + e_j}}{c_\beta c_{\beta + e_i + e_j}} \leq \frac{N(\beta+e_i)N(\beta+e_j)}{N(\beta)N(\beta +e_i+e_j)} \ \ \text{ for all } 1\le i,j\le d, \beta\in \N_0^d,$$ and is easily seen to equal the expression in Theorem \ref{2}.
\end{proof}


\section{Examples for Theorem \ref{2}}\label{SecExample2}
 In this Section we will present a large class of examples that satisfy the hypothesis of Theorem \ref{2}, and we will see that not all of them satisfy the hypothesis of Theorem \ref{1}. We will start very general and then move towards making it more concrete.

 Suppose that for $i=1, \dots d$ we are given sequences $S_i=\{s_{i,n}\}_{n\in \N_0}$ that are Kaluza sequences and satisfy the hypothesis of Theorem \ref{Kaluza} with $M=1$. Then define $\{c_\alpha\}_{\alpha\in \N_0^d}$ by
\begin{equation}\label{Example}c_\alpha = \dfrac{|\alpha|!}{\alpha!}\displaystyle{\prod_{i=1}^ds_{i,\alpha_i}}.\end{equation} We claim that any such $\{c_\alpha\}_{\alpha\in \N_0^d}$ satisfies the hypothesis of Theorem \ref{2}.

If $d=1$, then the hypothesis of Theorem \ref{2} is seen to be equivalent to the hypothesis of Theorem \ref{Kaluza} with $M=1$, so there is nothing to show. Assume now that $d\ge 2$. It is immediately clear   that $c_\textbf{0} = 1$ and that $0 < c_\alpha \leq \dfrac{|\alpha|!}{\alpha!}$.

To help us with the computations that are required to check the remaining condition, we introduce the notation $$P_i := \displaystyle{\prod_{\substack{k=1 \\ k \neq i}}^ds_{k,\alpha_k}}; \ \ \ \ \ Q_{i,j} := \displaystyle{\prod_{\substack{k=1 \\ k \neq i,j}}^ds_{k,\alpha_k}},$$where the latter is only defined for $i \neq j$. With this, we have for any $1 \leq i,j \leq d,\ i \neq j$: $$c_\alpha = \dfrac{|\alpha|!}{\alpha!}s_{i,\alpha_i}P_i; \ \ \ \ \ c_{\alpha + e_i} = \dfrac{(|\alpha|+1)!}{\alpha!(\alpha_i+1)}s_{i,\alpha_i+1}P_i; \ \ \ \ \ c_{\alpha + 2e_i} = \dfrac{(|\alpha|+2)!}{\alpha!(\alpha_i+1)(\alpha_i+2)}s_{i,\alpha_i+2}P_i$$ $$c_{\alpha+e_i+e_j} = \dfrac{(|\alpha|+2)!}{\alpha! (\alpha_i+1)(\alpha_j+1)}s_{i,\alpha_i+1}s_{j,\alpha_j+1}Q_{i,j}  \ \ \text{ if } i\ne j.$$ Subsequently, this leads to $$\dfrac{(c_{\alpha + e_i})^2}{c_\alpha c_{\alpha+2e_i}}= \left(\dfrac{(s_{i,\alpha_i+1})^2}{s_{i,\alpha_i}s_{i,\alpha_i+2}}\right)\left(\dfrac{(|\alpha|+1)(\alpha_i + 2)}{(|\alpha|+2)(\alpha_i+1)}\right) \leq \dfrac{(|\alpha|+1)(\alpha_i + 2)}{(|\alpha|+2)(\alpha_i+1)},$$using the hypothesis that $S_i$ is a Kaluza sequence for each fixed $i$. Also, for $i \neq j$, we have $$\dfrac{c_{\alpha + e_i}c_{\alpha + e_j}}{c_\alpha c_{\alpha + e_i + e_j}} = \left(\dfrac{P_j}{s_{i,\alpha_i}Q_{i,j}}\right)\left(\dfrac{|\alpha| + 1}{|\alpha| + 2}\right)= \dfrac{|\alpha| + 1}{|\alpha| + 2},$$and we have  shown $\{c_\alpha\}$ satisfies the hypothesis of Theorem \ref{2}.

Now let $\mu_1, \dots, \mu_d$ be probability measures on $[0,1]$ and define a measure on $[0,1]^d$ by $\mu= \mu_1 \times \dots \times \mu_d$. Then we claim that the coefficients of the function
\begin{equation}\label{fprodmeasure}f(z) =\int_{[0,1]^d} \frac{1}{1-\sum_{j=1}^d t_jz_j} d\mu(t)\end{equation} are of the type as considered above. Indeed, we have
\begin{align*}
\int_{[0,1]^d} \frac{1}{1-\sum_{j=1}^d t_jz_j} d\mu(t)&= \sum_{n=0}^\infty \int_{[0,1]^d} \left(\sum_{j=1}^d t_jz_j\right)^n d\mu(t)\\
&= \sum_{n=0}^\infty \sum_{|\alpha|=n} \frac{|\alpha|!}{\alpha!}z^\alpha \int_{[0,1]^d} t^\alpha d\mu(t)  \ \ \text{ by the Multinomial Theorem}\\
&= \sum_{n=0}^\infty \sum_{|\alpha|=n} \frac{|\alpha|!}{\alpha!} z^\alpha \prod_{i=1}^d \int_{[0,1]} t_i^{\alpha_i}d\mu_i(t_i)\\
&=  \sum_{\alpha\in \N_0^d} \frac{|\alpha|!}{\alpha!} z^\alpha \prod_{i=1}^d s_{i,\alpha_i},  \ \ \text{ where }s_{i,k}= \int_{[0,1]}t^kd\mu_i(t)\\
&= \sum_{\alpha\in \N_0^d} c_\alpha z^\alpha.
\end{align*}
It remains to note that $\{\int_{[0,1]} t^n d\sigma(t)\}_{n\in \N_0}$ is a Kaluza sequence for each probability measure $\sigma$ on $[0,1]$. In fact, this final step is well-known to be true,  basically it is even contained in Kaluza's paper. Namely, it follows easily from Hoelder's inequality that $F(s) := \displaystyle{\int_0^1t^sd\sigma(t)}$ defines a log-convex function on $[0,\infty)$ (also see \cite{Convexity}, Lemma 5.1).

By taking $d = 2$ and both $\mu_1$ and $\mu_2$ to be Lebesgue measure on $[0,1]$, we can show the resulting $\{c_\alpha\}$ do not satisfy the conditions of Theorem \ref{1}, namely that the sequence $r_\alpha$ is not non-decreasing. Let us write $(m,n)$ for arbitrary $\alpha \in \mathbb{N}^2_0$. We claim that $r_{(0,2)} > r_{(1,2)}$. To compute these quantities, we note that $s_{i,n} = \dfrac{1}{n+1}$ for $i$ = 1 and 2, and subsequently that $c_{(m,n)} = \dfrac{(m+n)!}{(m+1)!(n+1)!}$. We use this to calculate $c_{(0,1)} = c_{(1,1)} = c_{(1,2)} = 1/2$ and $c_{(0,2)} = 1/3$. From there, all we need to do is employ the definition of $r_{(m,n)}$ to see: $$r_{(0,2)} = \dfrac{c_{(0,2)}}{c_{(0,1)}} = 2/3 > 3/5 = \dfrac{c_{(1,2)}}{c_{(0,2)} + c_{(1,1)}} = r_{(1,2)}.$$

Note that there are probability measures on $[0,1]^2$ such that the function $f$ in (\ref{fprodmeasure}) does not satisfy the conclusion of  Theorems \ref{1} and \ref{2}. Let $0<a<1$, and for $w\in \mathbb{B}_2$ write $\delta_w$ for the unit point mass at $w$. Then
$  \mu= \frac{1}{2}(\delta_{(a,0)}+\delta_{(0,a)})$ is a probability measure, and
$$f(z,w)=\int_{[0,1]^2}\frac{1}{1-sz-tw}d\mu(s,t)= \frac{1}{2}\left(\frac{1}{1-az}+\frac{1}{1-aw}\right)=\frac{1-\frac{a}{2}(z+w)}{(1-az)(1-aw)}.$$ Then
\begin{align*}q(z,w)&=1-\frac{1}{f(z,w)}= \frac{1-\frac{a}{2}(z+w)-(1-az)(1-aw)}{1-\frac{a}{2}(z+w)}\\
&= \frac{\frac{a}{2}(z+w)-a^2zw}{1-\frac{a}{2}(z+w)}\\
&=\sum_{n=0}^\infty (\frac{a}{2})^{n+1}(z+w)^{n+1}- \sum_{n=0}^\infty a^2zw(\frac{a}{2})^{n}(z+w)^{n}.
\end{align*}
Thus, we determine the coefficient of $z^2w$ to be $q_{(2,1)}=\frac{a^3}{8}3 - \frac{a^3}{2}<0$.


\section{Applications to reproducing kernels}\label{sec:repkernels}
For $d\in \N$ let $\mathbb B_d$ denote the $\ell_2$-ball of $\C^d$, $\mathbb B_d=\{z\in \C^d: \sum_{j=1}^d |z_j|^2<1\}$. A Hilbert function space over $\Bd$ is a Hilbert space of functions $\Bd \to \C$ such that point evaluations for all $z\in \Bd$ are continuous. Each Hilbert function space has a reproducing kernel $k=\{k_w(z)\}$ that satisfies $k_w\in \HH$ for each $w\in \Bd$ and $f(w)=\la f, k_w\ra$ for each $f\in \HH$ and $w\in \Bd$.

\begin{defn} Let $k$ be a  reproducing kernel on $\Bd$.

(a) $k$  is called a normalized complete Nevanlinna-Pick kernel, if there is an auxiliary Hilbert space $\HK$ and a function $Q:\Bd \to \HK$ such that $Q(0)=0$ and $k_w(z)=\frac{1}{1-\la Q(z),Q(w)\ra}$ for all $z,w\in \Bd.$

(b) $k$ is called a normalized de Branges-Rovnyak kernel, if there  is an auxiliary Hilbert space $\HK$ and a function $B:\Bd \to \HK$ such that $B(0)=0$ and $k_w(z)=\frac{1-\la B(z),B(w)\ra}{1-\la z,w\ra}$ for all $z,w\in \Bd.$
\end{defn}
A Hilbert function space is called a normalized complete Nevanlinna-Pick space (normalized de Branges-Rovnyak space) if its reproducing kernel is a normalized complete Nevanlinna-Pick kernel (normalized de Branges-Rovnyak kernel).
Both of these types of kernels and spaces have been studied extensively in the literature.  For de Branges-Rovnyak spaces the results are most complete if $d=1$ and $\HK$ is one dimensional, see e.g. \cite{deBRo}, \cite{Sara}, \cite{FriMash1}, \cite{FriMash2}. If $d=1$ and $\HK$ is finite dimensional, then \cite{AleMal} and \cite{GuLuoR} contain some further results. Multivariable versions of de Branges spaces have been studied for their use in so-called transfer function realizations of analytic functions in $\Bd$, see e.g. \cite{BaBoFa1}, \cite{BaBoFa2}, but they have also been studied in their own right, \cite{Jury}, \cite{JuMa}.  For complete Nevanlinna-Pick kernels many results are available for all $d\in \N$ and all separable auxiliary Hilbert spaces $\HK$, we refer to \cite{AMcCpaper} and \cite{AMcC} for the basics. However these types of reproducing kernels are the subject of ongoing research, and there are many recent publications on the topic. The intuition is that normalized complete Nevanlinna-Pick spaces share many properties with the Hardy space $H^2$ of the unit disc and its multiplier algebra $H^\infty$. For example, there is a Pick interpolation theorem, a  commutant lifting theorem, a Beurling theorem,  a representation of the contractive multipliers as transfer functions of unitary colligations, and a theorem saying that every function in a normalized complete Nevanlinna-Pick space can be written as a ratio of multipliers of the space. It is thus of interest to determine whether a given reproducing kernel is a complete Nevanlinna-Pick kernel or whether a given Hilbert function space is a complete Nevanlinna-Pick space.

If $\HH\subseteq \Hol(\Bd)$ is a Hilbert function space such that the monomials $z^\alpha$ are contained in $\HH$ and  are mutually orthogonal, then for $f=\sum_{\alpha\in \N_0^d} \hat{f}(\alpha)z^\alpha$ we have $$\|f\|^2= \sum_{\alpha \in \N_0^d} |\hat{f}(\alpha)|^2 \|z^\alpha\|^2.$$ Since $z^\alpha \in \HH$ and since point evaluations are bounded, we must have $\|z^\alpha\|\ne 0$ and we see that the reproducing kernel is of the form
$$k_w(z)=\sum_{\alpha \in \N_0^d} c_\alpha \overline{w}^\alpha z^\alpha, \ \ c_\alpha= 1/\|z^\alpha\|^2.$$
 Thus, Theorems \ref{1} and \ref{2} give sufficient conditions for such a kernel to be a normalized complete Nevanlinna-Pick kernel, and Lemma \ref{lem:b alpha} and the discussion at the end of Section \ref{SecExample} show that the conditions of Theorem \ref{1} also are  sufficient  for $k_w(z)$ to be a normalized de Branges-Rovnyak kernel.

 We give an example to illustrate what the norms may look like when the kernel is a normalized complete Nevanlinna-Pick kernel of the form $$k_w(z)= \int_{[0,1]^d} \frac{1}{1-\sum_{j=1}^d t_j z_j\overline{w}_j} d\mu(t)$$ where $\mu=\mu_1 \times \dots \mu_d$ for probability measures on $[0,1]^d$. In \cite{AHMcR1} conditions on a probability measure $\mu$ on [0,1] were given that implied that  the norm on the space with reproducing kernel $$k_w(z) = \int_{[0,1]} \frac{1}{1-t\la z,w\ra}d\mu(t)$$ is equivalent to a weighted Besov-norm of the type $\|f\|^2=|f(0)|^2+\int_{\Bd}|R^\gamma f|^2 \omega dV$. Here $R$ is the radial derivative operator, $\gamma$ is a possibly fractional order, $dV$ is Lebesgue measure on $\Bd$, and $\omega$ is a radial weight on $\Bd$ (\cite{AHMcR1}, Lemma 5.1 and Theorem 5.2). Results and methods from \cite{AHMcR1} can be used to substantially generalize the following example, if one is willing to settle for equivalence rather than equality of norms.

 We take $d=2$ and $\mu_1=\mu_2=$ Lebesgue measure on $[0,1]$. Then in Section \ref{SecExample2} we calculated  $c_{(m,n)} = \dfrac{(m+n)!}{(m+1)!(n+1)!}$. If $F$ is a function of the complex variables $w=(w_1, w_2)$, then we write $$D_{w_j}F= \frac{\partial}{\partial w_j} w_jF.$$
If $f\in \Hol(\mathbb{B}_2)$, then define $Tf:\D\times \mathbb{B}_2\times \D^2\to \C$ by $Tf(\lambda,z,w)=f(\lambda z_1w_1,\lambda z_2w_2)$ and
\begin{equation}\label{BesovNorm}\|f\|^2 = \int_{\D} \int_{\mathbb{B}_2}\int_{\D^2} |D_\lambda  D_{w_1}D_{w_2}Tf(\lambda,z,w)|^2 \frac{dA(w_1)}{\pi}\frac{dA(w_2)}{\pi}d\sigma(z) \frac{dA(\lambda)}{\pi}.\end{equation} Here $dA$ denotes Lebesgue measure on $\D$ and $d\sigma$ is the rotationally invariant probability measure on $\mathbb{B}_2$.
It is easy to see that the monomials are orthogonal to one another in the inner product that comes from this norm.
If $f(z)=z^{(m,n)}$, then
 $Tf(\lambda,z,w)= \lambda^{m+n} z^{(m,n)}w^{(m,n)}$ and $$D_\lambda D_{w_1}D_{w_2}Tf= (m+n+1)(m+1)(n+1)\lambda^{m+n} z^{(m,n)}w^{(m,n)}.$$
 Now note  that \begin{align*}\int_\D|\lambda^{n+m}|^2\frac{dA(\lambda)}{\pi}&= \frac{1}{m+n+1},\\ \int_{\mathbb{B}_2} |z^{(m,n)}|^2d\sigma(z) &= \frac{m!n!}{(m+n+1)!},  \ \ \text{ and  }\\
 \int_\D\int_\D |w^{(m,n)}|^2 \frac{dA(w_1)}{\pi}\frac{dA(w_2)}{\pi}&= \frac{1}{(m+1)(n+1)}.\end{align*}

 Then
 \begin{align*} \|z^{(m,n)}\|^2&= (m+n+1)^2(m+1)^2(n+1)^2 \frac{1}{m+n+1} \frac{m!n!}{(m+n+1)!}\frac{1}{(m+1)(n+1)}\\
 &=\frac{(m+1)!(n+1)!}{(m+n)!}\\
 &=\frac{1}{c_{(m,n)}}.
 \end{align*}
Thus, if $\HH$ is the space of all holomorphic functions $f$ on $\mathbb{B}_2$, where the norm is given by $(\ref{BesovNorm})$, then the reproducing kernel is $$k_w(z) = \int_0^1\int_0^1\frac{1}{1-tz_1\overline{w}_1 - sz_2\overline{w}_2} dtds$$ and this is a complete Nevanlinna-Pick kernel.

\bibliography{Kaluza_references}
\end{document}